\documentclass[psamsfonts]{amsart}  

\usepackage[T1]{fontenc}
\usepackage{amssymb}

\usepackage[dvips]{graphicx}

\usepackage{graphicx}
\usepackage{amssymb}                       
\usepackage{amsmath}
\usepackage{color}
\usepackage{times}

\usepackage[unicode,bookmarks,colorlinks]{hyperref}
\hypersetup{
   linkcolor=brickred,
}

\definecolor{mahogany}{cmyk}{0, 0.77, 0.87, 0}
\definecolor{salmon}{cmyk}{0, 0.53, 0.38, 0}
\definecolor{melon}{cmyk}{0, 0.46, 0.50, 0}
\definecolor{yellowgreen}{cmyk}{0.44, 0, 0.74, 0}
\definecolor{brickred}{cmyk}{0, 0.89, 0.94, 0.28}
\definecolor{OliveGreen}{cmyk}{0.64, 0, 0.95, 0.40}
\definecolor{RawSienna}{cmyk}{0, 0.72, 1.0, 0.45}
\definecolor{ZurichRed}{rgb}{1, 0, 0} 

\usepackage{fancyhdr}
\usepackage{amsmath,amstext,amssymb,amsopn,amsthm}
\usepackage{amsmath,amssymb,amsthm}
\usepackage[mathscr]{eucal}

\reversemarginpar

\definecolor{rb}{rgb}{0.1,0.2, 0.7}

\begin{document}

\newtheorem{lemma}[thm]{Lemma}
\newtheorem{proposition}{Proposition}
\newtheorem{theorem}{Theorem}[section]
\newtheorem{deff}[thm]{Definition}
\newtheorem{case}[thm]{Case}
\newtheorem{prop}[thm]{Proposition}
\newtheorem{example}{Example}

\newtheorem{corollary}{Corollary}

\theoremstyle{definition}
\newtheorem{remark}{Remark}

\numberwithin{equation}{section}
\numberwithin{definition}{section}
\numberwithin{corollary}{section}

\numberwithin{theorem}{section}

\numberwithin{remark}{section}
\numberwithin{example}{section}
\numberwithin{proposition}{section}

\newcommand{\gap}{\lambda_{2,D}^V-\lambda_{1,D}^V}
\newcommand{\gapR}{\lambda_{2,R}-\lambda_{1,R}}
\newcommand{\bD}{\mathrm{I\! D\!}}
\newcommand{\calD}{\mathcal{D}}
\newcommand{\calA}{\mathcal{A}}

\newcommand{\conjugate}[1]{\overline{#1}}
\newcommand{\abs}[1]{\left| #1 \right|}
\newcommand{\cl}[1]{\overline{#1}}
\newcommand{\expr}[1]{\left( #1 \right)}
\newcommand{\set}[1]{\left\{ #1 \right\}}

\newcommand{\calC}{\mathcal{C}}
\newcommand{\calE}{\mathcal{E}}
\newcommand{\calF}{\mathcal{F}}
\newcommand{\Rd}{\mathbb{R}^d}
\newcommand{\BR}{\mathcal{B}(\Rd)}
\newcommand{\R}{\mathbb{R}}
\newcommand{\T}{\mathbb{T}}
\newcommand{\D}{\mathbb{D}}

\newcommand{\al}{\alpha}
\newcommand{\RR}[1]{\mathbb{#1}}
\newcommand{\bR}{\mathrm{I\! R\!}}
\newcommand{\ga}{\gamma}
\newcommand{\om}{\omega}
\newcommand{\A}{\mathbb{A}}
\newcommand{\bH}{\mathbb{H}}

\newcommand{\bb}[1]{\mathbb{#1}}
\newcommand{\bI}{\bb{I}}
\newcommand{\bN}{\bb{N}}

\newcommand{\uS}{\mathbb{S}}
\newcommand{\M}{{\mathcal{M}}}
\newcommand{\calB}{{\mathcal{B}}}

\newcommand{\W}{{\mathcal{W}}}

\newcommand{\m}{{\mathcal{m}}}

\newcommand {\mac}[1] { \mathbb{#1} }

\newcommand{\bC}{\Bbb C}

\newtheorem{rem}[theorem]{Remark}
\newtheorem{dfn}[theorem]{Definition}
\theoremstyle{definition}
\newtheorem{ex}[theorem]{Example}
\numberwithin{equation}{section}

\newcommand{\Pro}{\mathbb{P}}
\newcommand\F{\mathcal{F}}
\newcommand\E{\mathbb{E}}
\newcommand\e{\varepsilon}
\def\H{\mathcal{H}}
\def\t{\tau}

\title[Martingale inequalities]{A dual approach to Burkholder's $L^p$ estimates}

\author{Rodrigo Ba\~nuelos}
\thanks{R. Ba\~nuelos is supported in part  by NSF Grant  DMS-1854709}
\address{Department of Mathematics, Purdue University, West Lafayette, IN 47907, USA}
\email{banuelos@math.purdue.edu}

\author{Tomasz Ga\l \k{a}zka}
\address{Faculty of Mathematics, Informatics and Mechanics\\
 University of Warsaw\\
Banacha 2, 02-097 Warsaw\\
Poland}
\email{T.Galazka@mimuw.edu.pl}

\author{Adam Os\k{e}kowski}
\thanks{A. Os\k{e}kowski is supported by Narodowe Centrum Nauki (Poland), grant no. DEC-2014/14/E/ST1/00532.}\address{Faculty of Mathematics, Informatics and Mechanics\\
 University of Warsaw\\
Banacha 2, 02-097 Warsaw\\
Poland}
\email{A.Osekowski@mimuw.edu.pl}


\begin{abstract}
The paper contains an alternative proof of the celebrated $L^p$ estimates for differentially subordinate martingales established by Burkholder and Wang in the eighties and nineties. The approach links the validity of the estimate to the existence of a lower solution to a novel boundary value problem.
\end{abstract}

\maketitle

\section{Introduction}
Suppose that  $(\Omega,\mathcal{F},\mathbb{P})$ is a probability space, filtered by nondecreasing family $(\F_n)_{n\geq 0}$ of sub-$\sigma$-algebras of $\F$. Let $f=(f_n)_{n\geq 0}$, $g=(g_n)_{n\geq 0}$ be two martingales adapted to $(\mathcal{F}_n)$, taking values in a separable Hilbert space $\mathbb{H}$ with the norm $|\cdot|$ and the scalar product denoted by $\langle \cdot,\cdot\rangle$. Actually, we may and do assume that $\mathbb{H}=\ell^2$. Assume further that  $df=(df_n)_{n\geq 0}$, $dg=(dg_n)_{n\geq 0}$ are the difference sequences of $f$ and $g$, uniquely determined by the identities
$$ f_n=\sum_{k=0}^n df_k,\quad g_n=\sum_{k=0}^n dg_k,\quad n=0,\,1,\,2,\,\ldots.$$
Following Burkholder \cite{B0,B1}, we say that $g$ is differentially subordinate to $f$, if we have
\begin{equation}\label{differ}
 |dg_n|\leq |df_n|,\qquad n=0,\,1,\,2,\,\ldots,
\end{equation}
almost surely. For example, this is the case if $g$ is a transform of $f$ by a predictable real sequence $v=(v_n)_{n\geq 0}$, bounded in absolute value by $1$. That is,  for each $n$ we have $dg_n=v_n df_n$, $\|v_n\|_{L^\infty}\leq 1$ and $v_n$ is measurable with respect to $\F_{(n-1)\vee 0}$. 

The problem of comparing the sizes of $f$ and $g$ under the differential subordination has been studied in depth in the literature: we refer the reader to the papers \cite{B1,Os2} and the monograph \cite{Os} for the detailed description of this subject. We will concentrate on the following celebrated $L^p$-estimate, established by Burkholder in \cite{B0}. Here and below, we  use the notation $\|f\|_{L^p}=\sup_{n\geq 0}\|f_n\|_{L^p}$ for the $p$-th moment of a martingale $f$.

\begin{theorem}
Suppose that $1<p<\infty$. Then for any martingales $f$, $g$ such that $g$ is differentially subordinate to $f$ we have
\begin{equation}\label{moment}
||g||_{L^p}\leq (p^*-1) ||f||_{L^p},
\end{equation}
where $p^*=\max\{p,p'\}$ and $p'=p/(p-1)$ is the harmonic conjugate to $p$. The constant $p^*-1$ is the best possible.
\end{theorem}

This beautiful result has found many applications in various areas of mathematics, including harmonic analysis and geometric function theory. We refer the interested reader to \cite{AIG,BB,BMH,BO,NV} for more information. 

Burkholder's proof of \eqref{moment} rests on the construction of an appropriate special function. Suppose that $U:\mathbb{H}\times \mathbb{H}\to \R$ satisfies the following three conditions:

\smallskip

1$^\circ$ $U(x,y)\leq 0$ if $|y|\leq |x|$,

2$^\circ$ $U(x,y)\geq |y|^p-(p^*-1)^p|x|^p$,

3$^\circ$ $(U(f_n,g_n))_{n\geq 0}$ is a supermartingale for any pair $(f,g)$ of differentially subordinate martingales.

\smallskip

Then \eqref{moment} follows at once by observing that  $\E\big[|g_n|^p-(p^*-1)^p|f_n|^p\big]\leq \E U(f_n,g_n)\leq \E U(f_0,g_0)\leq 0$ and  letting $n\to \infty$. Burkholder \cite{B1} proved that the function
\begin{equation}\label{function_U}
 U(x,y)=p\left(1-\frac{1}{p^*}\right)^{p-1}(|y|-(p^*-1)|x|)(|x|+|y|)^{p-1}
\end{equation}
has all the required properties. See also \cite{B0} for a related approach.

There is a dual method of proving \eqref{moment}, developed by Nazarov, Treil and Volberg \cite{NT,NTV}, which is also based on the construction of a certain special function. Let us start with the case $p=2$, in which the description is particularly easy. Namely, consider the function $\mathbb{B}(x,z)=\frac{1}{2}(|x|^2+|z|^2)$. This function satisfies the following analogues of the above properties 1$^\circ$, 2$^\circ$ and 3$^\circ$:

\smallskip

1$^\circ$' $\mathbb B(x,z)\geq |xz|$,

2$^\circ$' $\mathbb B(x,z)\leq \frac{1}{2}(|x|^2+|z|^2)$ (actually, the equality holds here),

3$^\circ$' For any pair $(f,h)$ of \emph{arbitrary} martingales and any $n\geq 1$ we have 
\begin{equation}\label{dual_conv}
 \E \mathbb B(f_n,h_n)\geq \E \mathbb B(f_{n-1},h_{n-1})+\E |df_n||dh_n|.
\end{equation}

\smallskip

Let us stress here that the martingales $f$ and $h$ appearing in 3$^\circ$' are not related by any domination principle. These three conditions immediately give \eqref{moment} (for $p=2$). Indeed, fix an arbitrary pair $(f,g)$ satisfying the differential subordination and let $h$ be another martingale. By 1$^\circ$', the inductive use of 3$^\circ$' and finally 2$^\circ$', we get that for $n=0,\,1,\,2,\,\ldots,$
$$ \E \sum_{k=0}^n |df_k||dh_k|\leq  \E \mathbb B(f_0,h_0)+\E \sum_{k=1}^n |df_k||dh_k|\leq \E \mathbb B(f_n,h_n)\leq \frac{1}{2}(\E |f_n|^2+\E |h_n^2|).$$
Consequently, by the orthogonality of martingale differences and the differential subordination of $g$ to $f$, we obtain
\begin{align*}
 \E \langle g_n,h_n\rangle&=\E \sum_{k=0}^n \langle dg_k,dh_k\rangle\\
&\leq \E \sum_{k=0}^n |dg_k||dh_k|\leq \E \sum_{k=0}^n |df_k||dh_k|\leq \frac{1}{2}(\|f_n\|_{L^2}^2+\|h_n\|_{L^2}^2).
\end{align*}
By a standard homogenization argument, this gives $\E \langle g_n,h_n\rangle\leq \|f_n\|_{L^2}\|h_n\|_{L^2}$ which implies the desired bound $\|g_n\|_{L^2}\leq \|f_n\|_{L^2}$ by duality. There is a natural question whether this approach can be extended to other values of $p$. This problem was studied by Nazarov and Treil \cite{NT}. For $p>2$, they introduced the following function of \emph{four} variables: for $\zeta,\,\eta\in \mathbb{H}^2$,
$$ \widetilde{\mathbb{B}}(\zeta,\eta)=|\zeta|^p+|\eta|^q+\begin{cases}
|\zeta|^2|\eta|^{2-q} & \mbox{if }|\zeta|^p\geq |\eta|^{p'},\\
\displaystyle \frac{2}{p}|\zeta|^p+\left(\frac{2}{p'}-1\right)|\eta|^{p'} & \mbox{if }|\zeta|^p\leq |\eta|^{p'}.
\end{cases}$$
They showed that the function satisfies appropriate versions of 1$^\circ$, 2$^\circ$ and 3$^\circ$, which yields \eqref{moment} with a suboptimal constant. Despite its non-optimality, this special function has found many applications in harmonic analysis and semigroup theory; see e.g. \cite{DC,DV}. A further improvement is due to Ba\~nuelos and Os\k{e}kowski \cite{BO2}, who identified the appropriate version of $\widetilde{\mathbb{B}}$ leading to the best constant $p^*-1$ in the full range $1<p<\infty$. Still, this special function has four variables and the question remains. Can this approach be simplified to involve a function on $\mathbb{H}^2$, as in the above proof for the case $p=2$? The purpose of this paper is to answer this question in the affirmative and provide the explicit formula for the corresponding special functions.

Actually, we will study the above topic in the more general, continuous-time setting. Suppose $(\Omega,\F,\mathbb{P})$ is a complete probability space, equipped with a filtration $(\F_t)_{t\geq 0}$, such that $\F_0$ contains all the events of probability $0$. Let $X=(X_t)_{t\geq 0}$, $Y=(Y_t)_{t\geq 0}$ be two $\mathbb{H}$-valued local martingales, which have right-continuous paths with limits from the left. That is, a c\`adl\`ag martingales. The continuous-time extension of the differential subordination, which is due to Ba\~nuelos and Wang (see \cite{BW} and \cite{W}), can be formulated as follows. The process $Y$ is differentially subordinate to $X$, if the difference $([X,X]_t-[Y,Y]_t)$ is nonnegative and nondecreasing. Here $([X,X]_t)$ denotes the quadratic variance process of $X$: see Dellacherie and Meyer \cite{DM} for the definition in the real-valued case, and extend it to the vector context by $[X,X]_t=\sum_{j=1}^\infty [X^j,X^j]$, where $X^j$ is the $j$-th coordinate of $X$ (recall that we have assumed $\mathbb{H}=\ell^2$). This notion is a generalization of \eqref{differ}. To see this, note that if one treats discrete-time sequences $f,\, g$ as continuous time processes (via $X_t=f_{\lfloor t\rfloor}$ and $Y_t=g_{\lfloor t\rfloor}$), then the difference
$$ [X,X]_t-[Y,Y]_t=\sum_{k=0}^{\lfloor t \rfloor} \big( |df_k|^2-|dg_k|^2 \big)$$
is nonnegative and nondecreasing if and only if \eqref{differ} is valid.

It can be shown that essentially, all the results for discrete-time differentially subordinate (local) martingales carry over, with unchanged constants, to the continuous context: see the paper \cite{W} and the monograph \cite{Os2}. Roughly speaking, the transference can be described as follows.  As we have seen above, Burkholder's method of proving martingale inequalities involves the construction of a special function, satisfying appropriate size and concavity-type requirements. Once such a function is found, the continuous-time version follows from It\^o's formula and smoothing or stopping time arguments. In particular, this allowed Wang \cite{W} to obtain the following extension of \eqref{moment}. Here and below, we use the notation $\|X\|_{L^p}=\sup\|X_\tau\|_{L^p}$, where the supremum is taken over all finite stopping times.

\begin{theorem}
Suppose that $X$, $Y$ are $\mathbb{H}$-valued local martingales such that $Y$ is differentially subordinate to $X$. Then for any $1<p<\infty$ we have the sharp inequality
\begin{equation}\label{mainin0}
 \|Y\|_{L^p} \leq (p^*-1)\|X\|_{L^p}.
\end{equation}
\end{theorem}

Our contribution is to present a dual approach to this $L^p$ bound, in the spirit described above in the discrete-time setup. Namely, using a certain special function of two variables, we will establish the following estimate, which is easily seen to imply \eqref{mainin0} by duality arguments.

\begin{theorem}\label{mainthm}
Suppose that $X$, $Y$, $Z$ are $\mathbb{H}$-valued local martingales such that $Y$ is differentially subordinate to $X$. Then for any $1<p<\infty$ we have the sharp inequality
\begin{equation}\label{mainin}
 \|[Y,Z]_\infty\|_{L^1}\leq (p^*-1)\|X\|_{L^p}\|Z\|_{L^{p'}}.
\end{equation}
\end{theorem}

The paper is organized as follows. Our approach, reducing \eqref{mainin} to the search for an appropriate function, is described in the next section. In Section 3 we provide the explicit formulas for the special functions and prove that they enjoy all the required properties. In the last part we present some informal steps which led us to the discovery of the special function.

\section{On the approach}

The purpose of this section is to show that the existence of a certain special function on $\mathbb{H}^2$, or rather $\R^d\times \R^d$, implies the validity of \eqref{mainin}. Actually, all we need is an appropriate function defined on the first quadrant $(0,\infty)^2$. Namely, let $1<p<\infty$ and $K>0$ be fixed parameters. Suppose that $B:(0,\infty)^2\to \R$ is a function of class $C^2$, which enjoys the following properties:

\begin{itemize}
\item[1$^\circ$] (Initial condition) We have $B(x,z)\geq xz$.

\item[2$^\circ$] (Majorization) For any $x,\,z>0$ we have
\begin{equation}\label{maj}
B(x,z)\leq \frac{Kx^p}{p}+\frac{z^{p'}}{p'}.
\end{equation}

\item[3$^\circ$] (Monotonicity) For any $x,\,z>0$ we have 
\begin{equation}\label{monotonicity}
\frac{B_{xx}(x,z)}{|B_{xz}(x,z)|+1}\leq \frac{B_x(x,z)}{x}\quad \mbox{and}\quad \frac{B_{zz}(x,z)}{|B_{xz}(x,z)|+1}\leq \frac{B_z(x,z)}{z}.
\end{equation}

\item[4$^\circ$] (Concavity) For any $x,\,z>0$ and any $h,\,k\in \R$,
\begin{equation}\label{convex}
B_{xx}(x,z)h^2+2B_{xz}(x,z)hk+B_{zz}(x,z)k^2\geq 2|h||k|.
\end{equation}
\end{itemize}

Several observations are in order. First, we see that 1$^\circ$ and 2$^\circ$ are perfect analogues of the conditions 1$^\circ$', 2$^\circ$' appearing in the introduction. Furthermore, the requirements 3$^\circ$ and 4$^\circ$ should be treated as \emph{pointwise} conditions related to 3$^\circ$'. Next,  note that if $B$ satisfies 4$^\circ$, then we also have
\begin{equation}\label{convex2}
B_{xx}(x,z)h^2-2|B_{xz}(x,z)|\,|h||k|+B_{zz}(x,z)k^2\geq 2|h||k|
\end{equation}
(simply plug $h:=-\operatorname*{sgn}(B_{xz}(x,z)hk)h$ into \eqref{convex}). Consequently, we get
$$ B_{xx}(x,z)h^2-2(|B_{xz}(x,z)|+1)hk+B_{zz}(x,z)k^2\geq 0$$
for all $h,\,k\in \R$, which forces the corresponding discriminant to be nonpositive:
\begin{equation}\label{discriminant}
B_{xx}(x,z)B_{zz}(x,z)\geq (|B_{xz}(x,z)|+1)^2.
\end{equation}
Actually, it is easy to see that the implications can be reversed: the inequality \eqref{discriminant}, together with the inequality $B_{xx}(x,z)\geq 0$, implies the validity of \eqref{convex}. 

We are ready to introduce the special function $\mathbb{B}$, the extension of $B$ to higher dimensions. Namely, given $d\geq 2$, define $\mathbb{B}:(\R^d\times \R^d)\setminus \{(x,z):|x||z|=0\}\to \R$ by $ \mathbb{B}(x,y)=B(|x|,|y|).$ As we shall see, this object will lead us to the proof of \eqref{mainin}. In what follows, $\mathbb{B}_x$ and $\mathbb{B}_z$ will denote the vectors of partial derivatives of $\mathbb{B}$ with respect to the variables $x_1$, $x_2$, $\ldots$, $x_d$ and $z_1$, $z_2$, $\ldots$, $z_d$, respectively. Furthermore, the symbol $D^2\mathbb{B}$ will stand for the Hessian matrix of $\mathbb{B}$ and $\langle\cdot,\cdot\rangle$ will be the scalar product in $\R^d$.

\begin{lemma}
If $B$ satisfies 3$^\circ$ and 4$^\circ$, then for any $x,\,z\in \R^d\setminus \{0\}$ and any $h,\,k\in \R^d$ satisfying $x+h\neq 0$ and $z+k\neq 0$, we have
\begin{equation}\label{jump_control}
 \mathbb{B}(x+h,z+k)\geq \mathbb{B}(x,z)+\langle \mathbb{B}_x(x,z),h\rangle+\langle \mathbb{B}_z(x,z),k\rangle+ |h||k|.
\end{equation}
\end{lemma}
\begin{proof}
By continuity, we may assume assume that $x$ is not a multiple of $h$ and similarly that $z$ is not a multiple of $k$. That is, $x+th\neq 0$ and $z+tk\neq 0$ for all $t$. Consider the function $G(t)=\mathbb{B}(x+th,z+tk)$, given for $t\in \R$; then the assertion is equivalent to $G(1)\geq G(0)+G'(0)+|h||k|$. Observe that it is enough to show that $G''(t)\geq 2|h||k|$ for all $t$. Indeed, having proved this, we apply the mean value theorem to obtain $ G(1)-G(0)-G'(0)=\frac{1}{2}G''(t_0)$ for some intermediate number $t_0\in (0,1)$, and the claim follows.

So, fix $t\in \R$. Setting $x'=(x+th)/|x+th|$, $z'=(z+tk)/|z+tk|$ and $w=(|x+th|,|z+tk|)$, we compute that
\begin{equation}\label{bound_2}
\begin{split}
 G''(t)&=\langle D^2\mathbb B(x+th,z+tk)(h,k),(h,k)\rangle\\
&=\bigg\langle D^2B(w)\Big(\langle x',h\rangle,\langle z',k\rangle\Big),\Big(\langle x',h\rangle,\langle z',k\rangle\Big)\bigg\rangle\\
&\quad +B_x(w)\cdot \frac{|h|^2-\langle x',h\rangle^2}{|x+th|}+B_z(w)\cdot \frac{|k|^2-\langle z',k\rangle^2}{|z+tk|}\\
&=J_1+J_2+J_3+J_4,
\end{split}
\end{equation}
where
\begin{align*}
 J_1&= \frac{|B_{xz}(w)|}{|B_{xz}(w)|+1}B_{xx}(w)\langle x',h\rangle^2+2B_{xz}(w)\langle x',h\rangle\langle z',k\rangle+
\frac{|B_{xz}(w)|}{|B_{xz}(w)|+1}B_{zz}(w)\langle z',k\rangle^2,\\
J_2&=\left[\frac{B_x(w)}{|x+th|}-\frac{B_{xx}(w)}{|B_{xz}(w)|+1}\right](|h|^2-\langle x',h\rangle^2),\\
J_3&=\left[\frac{B_z(w)}{|z+tk|}-\frac{B_{zz}(w)}{|B_{xz}(w)|+1}\right]\!(|k|^2-\langle z',k\rangle^2),\\
J_4&=\frac{B_{xx}(w)}{|B_{xz}(w)|+1}|h|^2+\frac{B_{zz}(w)}{|B_{xz}(w)|+1}|k|^2.
\end{align*}
Let us analyze the terms $J_1$, $J_2$, $J_3$ and $J_4$. The first term is nonnegative.  To see this, note that $2B_{xz}(w)\langle x',h\rangle\langle z',k\rangle\geq -2|B_{xz}(w)\langle x',h\rangle\langle z',k\rangle|$ and hence it is enough to show that 
$$ B_{xx}(w)\langle x',h\rangle^2-2(|B_{xz}(w)|+1)|\langle x',h\rangle|\,|\langle z',k\rangle|+B_{zz}(w)\langle z',k\rangle^2\geq 0.$$
But this estimate follows directly from \eqref{convex2}. The terms $J_2,\,J_3$ are also nonnegative, which is an immediate consequence of 3$^\circ$. Finally, observe that  $J_4\geq 2|h||k|$: if we rewrite this in the equivalent form
$$ B_{xx}(w)|h|^2-2B_{xz}(w)|h||k|+B_{zz}(w)|k|^2\geq 2|h||k|,$$
we recognize \eqref{convex2} again. So, we have established the bound $G''(t)\geq 2|h||k|$ for all $t$.
\end{proof}

\begin{remark}
In particular, setting $t=0$ in \eqref{bound_2}, we obtain the estimate
\begin{equation}\label{bound_1}
\langle D^2\mathbb B(x,z)(h,k),(h,k)\rangle\geq \frac{B_{xx}(x,z)}{|B_{xz}(x,z)|+1}|h|^2+\frac{B_{zz}(x,z)}{|B_{xz}(x,z)|+1}|k|^2,
\end{equation}
which will be useful later.
\end{remark}

Here is the main result of this section, which links the special functions to the $L^p$ estimates for differentially subordinate (local) martingales.
\begin{theorem}
If there is a function $B:(0,\infty)^2\to \R$ satisfying 1$^\circ$-4$^\circ$, then \eqref{mainin} holds, with $p^*-1$ replaced by $K^{1/p}$.
\end{theorem}
\begin{proof}
By standard limiting arguments, it is enough to establish the desired estimate in the case when the processes take values in a finite- and at least two-dimensional subspace of $\mathbb{H}$.  That is, we may set $\mathbb{H}=\R^d$ for some $d\geq 2$. Let $X=(X_t)_{t\geq 0}$, $Z=(Z_t)_{t\geq 0}$ be arbitrary local martingales with values in $\R^d$. We may restrict ourselves to the case $\|X\|_{L^p}<\infty$ and $\|Z\|_{L^{p'}}<\infty$, since otherwise there is nothing to prove. Furthermore,  we may and do assume that $X$ and $Z$ are bounded away from zero, adding one dimension to $\mathbb{H}$ and modifying the processes slightly (if necessary). 

For any integer $N$, consider the stopping time $T_N=\inf\{t\geq 0:|[Y,Z]_t|\geq N\}$. By properties of stochastic integrals, the processes $\left(\int_{0+}^{t} \mathbb{B}_x(X_{s-},Z_{s-})\cdot\mbox{d}X_s\right)_{t\geq 0}$ and $\left(\int_{0+}^{ t} \mathbb{B}_z(X_{s-},Z_{s-})\cdot\mbox{d}Z_s\right)_{t\geq 0}$ are local martingales. Fix an arbitrary sequence $(\eta_n)_{n\geq 0}$ of stopping times which localizes these integrals and the processes $X$, $Z$. Fix $n$ and set $\tau_n=\eta_n\wedge T_N$. The application of It\^o's formula to the process $(\mathbb{B}(X_{\tau_n\wedge t},Z_{\tau_n\wedge t}))_{t\geq 0}$ yields
\begin{equation}\label{ito}
 \mathbb{B}(X_{\tau_n\wedge t},Z_{\tau_n\wedge t})=I_0+I_1+I_2/2+I_3,
\end{equation}
where
\begin{align*}
I_0&=\mathbb{B}(X_0,Z_0),\\
I_1&=\int_{0+}^{\tau_n\wedge t} \mathbb{B}_x(X_{s-},Z_{s-})\cdot\mbox{d}X_s+\int_{0+}^{\tau_n\wedge t} \mathbb{B}_z(X_{s-},Z_{s-})\cdot\mbox{d}Z_s,\\
I_2&=\int_{0+}^{\tau_n\wedge t} D^2\mathbb{B}(X_{s-},Z_{s-})d[X,Z]_s^c,\\
I_3&=\sum_{0<s\leq {\tau_n\wedge t}} \bigg[\mathbb{B}(X_s,Z_s)-\mathbb{B}(X_{s-},Z_{s-})- \mathbb{B}_x(X_{s-},Z_{s-})\Delta X_s- \mathbb{B}_z(X_{s-},Z_{s-})\Delta Z_s\bigg].
\end{align*}
Let us analyze the behavior of these terms. First, by 1$^\circ$ and the differential subordination of $Y$ to $X$, we have
$$ I_0\geq |X_0||Z_0|\geq |Y_0||Z_0|=[Y,Z]_0.$$
Next, we have $\E I_1=0$, by the properties of stochastic integrals. To deal with $I_2$, fix $0\leq s_0<s_1\leq t$. For any $\ell\geq 0$, let $(\eta_i^\ell)_{0\leq i\leq i_\ell}$ be a nondecreasing sequence of stopping times with $\eta_0^\ell=s_0,\,\eta_{i_\ell}^\ell=s_1$ such that $\lim_{\ell\to \infty}\max_{0\leq i\leq i_\ell-1}|\eta_{i+1}^\ell-\eta_i^\ell|=0$. Keeping $\ell$ fixed, we apply, for each $i=0,\,1,\,2,\,\ldots,\,i_\ell$, the estimate \eqref{bound_1} with $x=X_{s_0-}$, $z=Z_{s_0-}$ and $h=X^{c}_{\eta_{i+1}^\ell}-X^{c}_{\eta_{i}^\ell}$,  $k=Z^{c}_{ \eta_{i+1}^\ell}-Z^{c}_{\eta_{i}^\ell}$. We sum the obtained $i_\ell+1$ inequalities and let $\ell\to \infty$. As a result, we obtain the estimate
\begin{align*}
& \int_{s_0+}^{s_1}D^2\mathbb{B}(X_{s_0-},Z_{s_0-})\mbox{d}[X,Z]_s^c\\
&\geq \int_{s_0+}^{s_1}\frac{B_{xx}(X_{s_0-},Z_{s_0-})}{|B_{xz}(X_{s_0-},Z_{s_0-})|+1}\mbox{d}[X,X]^c_s+\int_{s_0+}^{s_1}\frac{B_{zz}(X_{s_0-},Z_{s_0-})}{|B_{xz}(X_{s_0-},Z_{s_0-})|+1}\mbox{d}[Z,Z]^c_s\\
&\geq \int_{s_0+}^{s_1}\frac{B_{xx}(X_{s_0-},Z_{s_0-})}{|B_{xz}(X_{s_0-},Z_{s_0-})|+1}\mbox{d}[Y,Y]^c_s+\int_{s_0+}^{s_1}\frac{B_{zz}(X_{s_0-},Z_{s_0-})}{|B_{xz}(X_{s_0-},Z_{s_0-})|+1}\mbox{d}[Z,Z]^c_s,
\end{align*}
where in the last passage we have used the differential subordination of $Y^c$ to $X^c$. By the Kunita-Watanabe inequality and the estimate \eqref{discriminant}, the latter expression is not smaller than $2 \int_{s_0+}^{s_1}d[Y,Z]^c_s=[Y,Z]^c_{s_1}-[Y,Z]^c_{s_0}$ and hence, approximating $I_2$ by Riemann sums, we obtain $I_2\geq 2[Y,Z]^c_{\tau_n\wedge t}-2[Y,Z]_0^c=2[Y,Z]^c_{\tau_n\wedge t}$. Finally, the term $I_3$ is handled with the use of \eqref{jump_control}: we get
$$ I_3\geq \sum_{0<s\leq \tau_n\wedge t}|\Delta X_s||\Delta Z_s|\geq \sum_{0<s\leq \tau_n\wedge t}|\Delta Y_s||\Delta Z_s|=\sum_{0<s\leq \tau_n\wedge t}\langle \Delta Y_s, \Delta Z_s\rangle.$$
Plugging all these observations into \eqref{ito} gives
\begin{equation}\label{dla}
 \begin{split}
\mathbb{B}(X_{\tau_n\wedge t},Z_{\tau_n\wedge t})&\geq |Y_0||Z_0|+I_1+[Y,Z]^c_{\tau_n\wedge t}+\sum_{0< s\leq \tau_n\wedge t}\langle \Delta Y_s, \Delta Z_s\rangle\\
&\geq I_1+[Y,Z]_{\tau_n\wedge t}.
\end{split}
\end{equation}
The expressions above are integrable: by 1$^\circ$ and 2$^\circ$, we have
$$ 0\leq \mathbb{B}(X_{\tau_n\wedge t},Z_{\tau_n\wedge t})\leq \frac{K|X_{\tau_n\wedge t}|^p}p+\frac{|Z_{\tau_n\wedge t}|^{p'}}{p'},$$
and the right-hand side is integrable, since $\|X\|_{L^p}<\infty$ and $\|Z\|_{L^{p'}}<\infty$ (see the beginning of the proof). Furthermore, by the very definition of $T_N$ and the differential subordination
$$|[Y,Z]_{\tau_n\wedge t}|\leq |[Y,Z]_{\tau_n\wedge t-}|+|\Delta Y_{\tau_n\wedge t}||\Delta Z_{\tau_n\wedge t}|\leq N+|\Delta X_{\tau_n\wedge t}||\Delta Z_{\tau_n\wedge t}|,$$
and the latter expression is integrable by the Young inequality and the $L^p$ / $L^{p'}$-boundedness of $X$ and $Z$. Consequently, taking the expectation in \eqref{dla}, recalling that $\E I_1=0$, and applying the majorization condition 2$^\circ$, we obtain
$$ \E [Y,Z]_{\tau_n\wedge t}\leq \frac{K\E |X_{\tau_n\wedge t}|^p}{p}+\frac{\E |Z_{\tau_n\wedge t}|^{p'}}{p'}\leq \frac{K\|X\|_{L^p}^p}{p}+\frac{\|Z\|_{L^{p'}}^{p'}}{p'},$$
or, by a simple homogenization argument,
$$ \E [Y,Z]_{\tau_n\wedge t}\leq K^{1/p}\|X\|_{L^p}\|Z\|_{L^{p'}}.$$
Letting $n\to \infty$, $t\to\infty$, $N\to \infty$ and using standard limit theorems, we get the desired estimate \eqref{mainin} (with $p^*-1$ replaced by $K^{1/p}$).
\end{proof}

As a by-product, we obtain the following interesting estimate for the joint variation of $X$ and $Z$.

\begin{remark}
For any $T>0$, let $ [[X,Z]]_T=\operatorname*{limsup} \sum_{j=1}^N |X_{t_j}-X_{t_{j-1}}||Z_{t_j}-Z_{t_{j-1}}|$, where the limit is taken over all $N$ and all partitions $0=t_0<t_1<t_2<\ldots<t_N=T$ of $[0,T]$ with diameter tending to $0$. Let $[[X,Z]]_\infty=\lim_{T\to\infty}[[X,Z]]_T$ be the total joint variation of the pair $(X,Z)$. 
The above reasoning can be easily adapted to yield the estimate
\begin{equation}\label{stronger_assertion}
 \E [[X,Z]]_\infty\leq K^{1/p}\|X\|_{L^p}\|Z\|_{L^{p'}}.
\end{equation}
Indeed, when handling the terms $I_2$ and $I_3$, skip all the arguments which involve the martingale $Y$ and the differential subordination.
\end{remark}

\textbf{An application.} Before we proceed to the construction of appropriate special functions, let us say a few words about possible applications of the estimate \eqref{stronger_assertion}. Suppose that $W=(W_t)_{t\geq 0}$ is a standard Brownian motion in $\R^d$ and let $H=(H_t)_{t\geq 0}$, $K=(K_t)_{t\geq 0}$ be two predictable processes also taking values in $\R^d$. Then the stochastic integrals
$$
X_t=\int_0^t H_s\cdot \mbox{d}W_s, \,\,\,\, Y_t=\int_0^t K_s\cdot \mbox{d}W_s,\qquad t\geq 0,
$$
are martingales and hence we have
\begin{equation}\label{integral_inequality}
 \E\int_0^{\infty} |H_s||K_s| ds \leq K^{1/p}\|X\|_{L^{p}}\|Y\|_{L^{p'}}\leq \frac{K\|X\|_{L^p}^p}{p}+\frac{\|Y\|_{L^{p'}}^{p'}}{p'}.
\end{equation}

Now, consider arbitrary sufficiently regular  functions $f\in L^p(\R^d)$ and $g\in L^{p'}(\R^d)$. Let $u_f$, $u_g$ stand for the corresponding heat extensions  to the upper half-space:  $u_f(t,x)=P_t f(x)$ and $u_g(t,x)=P_t g(x)$, where $(P_t)_{t\geq 0}$ is the semigroup with the kernel
$$ p_t(x,y)=(2\pi t)^{-d/2}\exp(-|x-y|^2/2t),\qquad x,\,y\in \R^d,\,t>0.$$
Then $u_f$, $u_g$ satisfy the heat equation in the interior the halfspace and hence for any fixed $T>0$ and $x\in \R^d$, the processes $X^{T,x}=(u_f(T-t,x+W_t))_{0\leq t\leq T}$, $Y^{T,x}=(u_g(T-t,x+W_t))_{0\leq t\leq T}$ are martingales. In addition, It\^o's formula yields the representations
$$ X_t^{T,x}=u_f(T,x)+\int_0^t \nabla_x u_f(T-s,x+W_s)\cdot \mbox{d}W_s,$$
$$Y_t^{T,x}=u_g(T,x)+\int_0^t \nabla_x u_g(T-s,x+W_s)\cdot \mbox{d}W_s.$$
Thus, an application of \eqref{integral_inequality} gives
\begin{align*}
 &\E \int_0^T |\nabla_x u_f(T-s,x+W_s)||\nabla_x u_g(T-s,x+W_s)|\mbox{d}s\\
&\qquad \qquad \leq \frac{K\E|f(x+W_T)|^p}{p}+\frac{\E|g(x+W_T)|^{p'}}{p'}.
\end{align*}
Integrating both sides over $x\in \R^d$ (with respect to the Lebesgue measure) and using Fubini's theorem, we obtain
$$ \int_0^T \int_{\R^d} |\nabla_x u_f(T-s,x)||\nabla_x u_g(T-s,x)|\mbox{d}x\mbox{d}s\leq \frac{K\|f\|_{L^p(\R^d)}^p}{p}+\frac{\|g\|_{L^{p'}(\R^d)}^{p'}}{p'}.$$
Hence, changing the variables $s:=T-s$ on the left, letting $T\to \infty$ and applying a homogenization argument, we get the Littlewood-Paley-type inequality
$$ \int_0^{\infty}\int_{\R^d} |\nabla_x u_f(x, t)||\nabla_x u_g(x, t)|dxdt\leq K^{1/p}\|f\|_{L^{p}(\R^d)}\|g\|_{L^{p'}(\R^d)}.$$
By a simple approximation argument, this extends to general $f\in L^p(\R^d)$ and $g\in L^{p'}(\R^d)$, without any additional regularity assumptions. 
A similar reasoning, which exploits the Poisson semigroup instead of $(P_t)_{t\geq 0}$ and the stopped Brownian motion in $[0,\infty)\times \R^d$ instead of $((T-t,x+B_t))_{0\leq t\leq T}$, yields the corresponding estimate
$$
\int_0^{\infty}\int_{\R^d} 2y|\nabla v_f(x, y)||\nabla v_g(x, y)|dxdy\leq  K^{1/p}\|f\|_{L^{p}(\R^d)}\|g\|_{L^{p'}(\R^d)},
$$
where $v_f$, $v_g$ denote the Poisson extensions of $f$ and $g$ to the upper halfspace.

Similar inequalities hold for other semigroups including those arising from nonlocal operators.  As an example,  consider the semigroup $(P_t)_{t\geq 0}$ arising from the process of  a L\'evy measure $\nu$ under the assumption in \cite{BaBoLu}.  Then 
\begin{eqnarray*}
\int_{\R^d} \int_0^{\infty}\int_{\R^d}|P_tf(x+y)-P_tf(x)|\,|P_tg(x+y)-P_tg(x)|\ \nu(dy)dtdx\\
\leq K^{1/p}\|f\|_{L^{p}(\R^d)}\|g\|_{L^{p'}(\R^d)}.
 \end{eqnarray*}   
For details on how the martingales arise in this case, see  \cite[pp 470-471]{BaBoLu}.

\section{Explicit special functions}

In the light of the reasoning from the previous section, the inequality \eqref{mainin} will follow if we construct a special function $B$ with $K=(p^*-1)^p$. The case $p=2$ has  already been dealt with in the introduction. We consider the cases $1<p<2$ and $p>2$ separately.

\subsection{The case $1<p<2$}

We start with the introduction of a certain auxiliary function.

\begin{lemma}\label{phi<2}
For any $s\geq 0$, there is a unique positive number $\varphi=\varphi(s)$ satisfying
\begin{equation}\label{def_phi<2}
 \varphi(s)(1+\varphi(s))^{p-2}=p^{p-2}s.
\end{equation}
The resulting function $\varphi:[0,\infty)\to [0,\infty)$ is of class $C^\infty$ and satisfies
\begin{equation}\label{diff_phi<2}
\varphi'(s)=\frac{\varphi(s)(1+\varphi(s))}{s(1+(p-1)\varphi(s))},\qquad s>0.
\end{equation}
In addition, we have $\varphi(s)\geq s^{1/(p-1)}$ for $s\leq (p-1)^{1-p}$ and $\varphi(s)<s^{1/(1-p)}$ for $s>(p-1)^{1-p}$.
\end{lemma}
\begin{proof}
The existence and uniqueness of $\varphi(s)$ follows at once from the fact that the function $\Phi(u)=u(1+u)^{p-2}$ satisfies $\Phi(0)=0$ and $\Phi'(u)=(1+u)^{p-3}(1+(p-1)u)>0$ for $u>0$ and $\Phi(u)\to \infty$ as $u\to \infty$. The differentiability of $\varphi$ is an immediate consequence of standard theorems on implicit functions.  To show the identity \eqref{diff_phi<2}, it suffices to differentiate both sides of \eqref{def_phi<2} and rearrange terms. To prove the final part of the lemma, we invoke the monotonicity property of $\Phi$ described above. Namely, if $s\leq (p-1)^{1-p}$, then we have
$$ \Phi(s^{1/(p-1)})=s^{1/(p-1)}(1+s^{1/(p-1)})^{p-2}\leq s^{1/(p-1)}\big((p-1)s^{1/(p-1)}+s^{1/(p-1)}\big)^{p-2}=p^{p-2}s$$
and hence $\varphi(s)\geq s^{1/(p-1)}$. In the case $s>(p-1)^{1-p}$, we just reverse the estimates in the above reasoning.
\end{proof}

The central object, the special function $B:(0,\infty)^2\to \R$,  is defined by
$$ B(x,z)=xz\left[\frac{p-1}{p}\varphi(x^{1-p}z)+\frac{2-p}{p(p-1)}+\frac{1}{p(p-1)\varphi(x^{1-p}z)}\right].$$

\begin{theorem}
The function $B$ satisfies the conditions 1$^\circ$-4$^\circ$ listed in the previous section.
\end{theorem}
\begin{proof}[Proof of 1$^\circ$] This is easy: we have $as+bs^{-1}\geq 2\sqrt{ab}$,  for any $a,\,b,\,s>0$, so
$$ \frac{p-1}{p}\varphi(x^{1-p}z)+\frac{2-p}{p(p-1)}+\frac{1}{p(p-1)\varphi(x^{1-p}z)}\geq \frac{2-p}{p(p-1)}+\frac{2}{p}=\frac{1}{p-1}\geq 1. \qedhere $$
\end{proof}
\begin{proof}[Proof of 3$^\circ$ and 4$^\circ$]
 As before, to keep the notation short, we will set $s=x^{1-p}z$. We compute directly that
\begin{align*}
&B_x(x,z)\\
&=z\left[\frac{p-1}{p}\varphi(s)+\frac{2-p}{p(p-1)}+\frac{1}{p(p-1)\varphi(s)}\right]+sz\left[-\frac{(p-1)^2}{p}\varphi'(s)+\frac{\varphi'(s)}{p\varphi^2(s)}\right].
\end{align*}
The second expression on the right-hand side equals
$$ \frac{sz\varphi'(s)}{p\varphi^2(s)}(1-(p-1)\varphi(s))(1+(p-1)\varphi(s))=\frac{z(1+\varphi(s))(1-(p-1)\varphi(s))}{p\varphi(s)},$$
where the last equality is due to \eqref{diff_phi<2}. Consequently,
$$ B_x(x,z)=\frac{(2-p)z}{p-1}+\frac{z}{(p-1)\varphi(s)}.$$
Next, we see that 
$$ B_{xx}(x,z)=\frac{x^{-p}z^2\varphi'(s)}{\varphi^2(s)}$$
and, exploiting \eqref{diff_phi<2} again we have that,
\begin{align*}
 B_{xz}(x,z)&=-\frac{p-2}{p-1}+\frac{1}{(p-1)\varphi(s)}-\frac{s\varphi'(s)}{(p-1)\varphi^2(s)}\\
 &=-\frac{p-2}{p-1}+\frac{1}{(p-1)\varphi(s)}-\frac{1+\varphi(s)}{(p-1)\varphi(s)(1+(p-1)\varphi(s))}\\
 &=-\frac{p-2}{p-1}+\frac{p-2}{(p-1)(1+(p-1)\varphi(s))}\\
 &=\frac{(2-p)\varphi(s)}{1+(p-1)\varphi(s)}.
\end{align*}
Finally, note that 
\begin{align*}
& B_z(x,z)\\
&=x\left[\frac{p-1}{p}\varphi(s)+\frac{2-p}{p(p-1)}+\frac{1}{p(p-1)\varphi(s)}\right]+x\left[\frac{(p-1)s}{p}\varphi'(s)-
\frac{s\varphi'(s)}{p(p-1)\varphi^2(s)}\right].
\end{align*}
By  \eqref{diff_phi<2}, the second term on the right-hand side equals
$$ \frac{xs\varphi'(s)}{p(p-1)\varphi^2(s)}((p-1)\varphi(s)+1)((p-1)\varphi(s)-1)=\frac{x(1+\varphi(s))((p-1)\varphi(s)-1)}{p(p-1)\varphi(s)}.$$
This, after some straightforward manipulations, yields $B_z(x,z)=x\varphi(s)$. In addition, we immediately obtain $B_{zz}(x,z)=x^{2-p}\varphi'(s)$. 

We  are now ready to check 3$^\circ$ and 4$^\circ$. Note that
\begin{align*}
 \frac{B_{xx}(x,z)}{|B_{xz}(x,z)|+1}-\frac{B_x(x,z)}{x}&=\frac{z}{x}\left[\frac{s\varphi'(s)(1+(p-1)\varphi(s))}{\varphi^2(1+\varphi)}-\frac{1+(2-p)\varphi(s)}{(p-1)\varphi(s)}\right]\\
&=\frac{(p-2)z(1+\varphi(s))}{x(p-1)\varphi(s)}\leq 0
\end{align*}
and 
$$ \frac{B_{zz}(x,z)}{|B_{xz}(x,z)|+1}-\frac{B_z(x,z)}{z}=\frac{x}{z}\left[\frac{s\varphi'(s)(1+(p-1)\varphi(s))}{1+\varphi}-\varphi(s)\right]=0.$$
Thus, 3$^\circ$ holds true. To check the concavity condition \eqref{convex}, note that $B_{xx}(x,z)\geq 0$ and hence it is enough to check the discriminant inequality \eqref{discriminant}. 
This estimate reads
$$ \frac{s^2(\varphi'(s))^2}{\varphi^2(s)}\geq \left(\frac{1+\varphi(s)}{1+(p-1)\varphi(s)}\right)^2,$$
and follows from \eqref{diff_phi<2}: actually, both sides are equal.
\end{proof}

\begin{proof}[Proof of 2$^\circ$] It remains to handle the majorization property, with $K=(p^*-1)^p=(p-1)^{-p}$.  
The claim is equivalent to
$$ \frac{p-1}{p}s\varphi(s)+\frac{2-p}{p(p-1)}s+\frac{s}{p(p-1)\varphi(s)}-\frac{(p-1)s^{p/(p-1)}}{p}\leq \frac{1}{p(p-1)^p},$$
where, as previously, $s=x^{1-p}z$. Denote the left-hand side of the above estimate by $L(s)$. Differentiating, we see that  $ L'(s)=B_z(1,s) -s^{1/(p-1)}=\varphi(s)-s^{1/(p-1)}.$ Therefore,  by the last part of Lemma  \ref{phi<2}, we conclude that $L$ attains its maximum at the point $(p-1)^{1-p}$. It remains to check that $L((p-1)^{1-p})=p^{-1}(p-1)^{-p}$.  This follows easily from the identity $\varphi((p-1)^{1-p})=s^{1/(p-1)}$, which can be verified directly in \eqref{def_phi<2}.
\end{proof}

\subsection{The case $p>2$} We proceed in a similar manner, starting with an auxiliary function.

\begin{lemma}\label{phi>2}
For any $s\geq 0$, there is a unique number $\varphi=\varphi(s)\in [p-2,\infty)$ satisfying
\begin{equation}\label{def_phi>2}
p\left(1-\frac{1}{p}\right)^{p-1}(1+\varphi(s))^{p-2}(\varphi(s)-p+2)=s.
\end{equation}
The resulting function $\varphi:[0,\infty)\to [p-2,\infty)$ is of class $C^\infty$ and satisfies
\begin{equation}\label{diff_phi>2}
\varphi'(s)=\frac{(1+\varphi(s))(\varphi(s)-p+2)}{(p-1)s(\varphi(s)-p+3)},\qquad s>0.
\end{equation}
In addition, we have $\varphi(s)\geq s^{1/(p-1)}$ for $s\leq (p-1)^{p-1}$ and $\varphi(s)<s^{1/(1-p)}$ for $s>(p-1)^{p-1}$.
\end{lemma}

The special function $B$ is given by
$$ B(x,z)=\left(1-\frac{1}{p}\right)xz\left[\varphi(x^{1-p}z)+\frac{1}{\varphi(x^{1-p}z)-p+2}\right].$$

In the light of the previous section, the inequality \eqref{mainin} will follow once we show the following.
\begin{theorem}
The function $B$ satisfies the conditions 1$^\circ$-4$^\circ$.
\end{theorem}
\begin{proof}[Proof of 1$^\circ$]
We have $s+s^{-1}\geq 2$ for any $s>0$, so
$$ \left(1-\frac{1}{p}\right)\left[\varphi(x^{1-p}z)+\frac{1}{\varphi(x^{1-p}z)-p+2}\right]\geq \left(1-\frac{1}{p}\right)(2+p-2)=p-1\geq 1. \qedhere$$
\end{proof}

\begin{proof}[Proof of 3$^\circ$ and 4$^\circ$]
The calculation are similar to those in the case $1<p<2$. A direct differentiation combined with \eqref{diff_phi>2} yields
\begin{align*}
B_x(x,z)&=\!\left(1-\frac{1}{p}\right)z\!\left[\varphi(s)+\frac{1}{\varphi(s)-p+2}\right]\!-\frac{(p-1)^2}{p}sz\left[\varphi'(s)-\frac{\varphi'(s)}{(\varphi(s)-p+2)^2}\right]\\
&=\left(1-\frac{1}{p}\right)z\left[\varphi(s)+\frac{1}{\varphi(s)-p+2}\right]-\frac{(p-1)z(1+\varphi(s))(\varphi(s)-p+1)}{p(\varphi(s)-p+2)}\\
&=\frac{(p-1)z}{\varphi(s)-p+2}.
\end{align*}
Therefore,
$$ B_{xx}(x,z)=\frac{(p-1)^2 x^{-p}z^2\varphi'(s)}{(\varphi(s)-p+2)^2}$$
and, by \eqref{diff_phi>2},
\begin{align*}
 B_{xz}(x,z)&=\frac{p-1}{\varphi(s)-p+2}-\frac{(p-1)s\varphi'(s)}{(\varphi(s)-p+2}\\
 &=\frac{p-1}{\varphi(s)-p+2}-\frac{1+\varphi(s)}{(\varphi(s)-p+2)(\varphi(s)-p+3)}=\frac{p-2}{\varphi(s)-p+3}.
\end{align*}
Moreover, arguing as above, we check that $B_z(x,z)$ equals
\begin{align*}
\left(1-\frac{1}{p}\right)x\left[\varphi(s)+\frac{1}{\varphi(s)-p+2}\right]+\left(1-\frac{1}{p}\right)sx\left[\varphi'(s)-\frac{\varphi'(s)}{(\varphi(s)-p+2)^2}\right]=x\varphi(s)
\end{align*}
and hence $B_{zz}(x,z)=x^{2-p}\varphi'(s)$. 
Now we can establish 3$^\circ$.  We have
$$ \frac{B_{xx}(x,z)}{|B_{xz}(x,z)|+1}-\frac{B_x(x,z)}{x}\!=\!\frac{(p-1)z}{x(\varphi(s)-p+2)}\left[\frac{(p-1)s\varphi'(s)(\varphi(s)-p+3)}{(\varphi(s)-p+2)(1+\varphi(s))}-1\right]=0$$
and
$$ \frac{B_{zz}(x,z)}{|B_{xz}(x,z)|+1}-\frac{B_z(x,z)}{z}=\frac{(2-p)x(\varphi(s)+1)}{(p-1)z}\leq 0.$$
To check the concavity 4$^\circ$, it is enough to verify the validity of \eqref{discriminant} (since $B_{xx}(x,z)\geq 0$). However, we have
$$ B_{xx}(x,z)B_{zz}(x,z)=\left(\frac{(p-1)s\varphi'(s)}{\varphi(s)-p+2}\right)^2,$$
which, by \eqref{diff_phi>2}, is equal to $\displaystyle \left(\frac{1+\varphi(s)}{\varphi(s)-p+2}\right)^2=(|B_{xz}(x,z)|+1)^2$. 
\end{proof}

\begin{proof}[Proof of 2$^\circ$]
We proceed as in the previous case. We let $K=(p^*-1)^p=(p-1)^p$ and note that the majorization is equivalent to
$$ \left(1-\frac{1}{p}\right)s\left[\varphi(s)+\frac{1}{\varphi(s)-p+2}\right]-\frac{p-1}{p}s^{p/(p-1)}\leq \frac{(p-1)^p}{p}.$$
Denoting the left-hand side by $L(s)$, we compute that $ L'(s)=B_z(1,s)-s^{1/(p-1)}=\varphi(s)-s^{1/(p-1)}$. By the last part of Lemma \ref{phi>2}, $L$ attains its maximal value at $s=(p-1)^{p-1}$. Since $\varphi((p-1)^{p-1})=p-1$ (directly from \eqref{def_phi>2}), we check that $L((p-1)^{p-1})=(p-1)^p/p$, which establishes the desired majorization.
\end{proof}

\section{On the search of the Bellman function $B$}

Now we will sketch some informal steps which lead to the discovery of the special functions $B$ (and the optimal constant $K=(p^*-1)^p$); the reasoning will be based on a number of assumptions and guesses. A typical approach during the search for the Bellman function is to look  at the concavity condition and assume its degeneracy. This usually gives rise to a corresponding second order partial differential equation. Next, one exploits structural properties of a general solution of the equation and from this one aims to come up with a reasonable candidate for the special function. 
We consider the cases $1<p<2$ and $p>2$ separately.

\subsection{The case $1<p<2$} It is convenient to split the argumentation into a few steps.

\smallskip

\emph{Step 1. Additional assumptions}. We will impose a few extra assumptions on the function $B$. First, we guess that the partial derivative $B_{xz}$ is \emph{nonnegative} on the whole $(0,\infty)^2$. Our second assumption is that $B$ can be extended to a function $\tilde{B}:\R^2\to \R$ satisfying $\tilde{B}(x,z)=\tilde{B}(\pm x,\pm z)$. Then, in particular, $\tilde B$ must satisfy 
\begin{equation}\label{diff_B}
\tilde B_z(x,0)=0.
\end{equation} 

\smallskip

\emph{Step 2. The Monge-Amp\`ere equation}.  In our case, the concavity is governed by the inequality \eqref{discriminant}. Setting $C(x,z)=\tilde{B}(x,z)+xz$ and recalling the assumption $B_{xz}>0$ for $x,\,z>0$, we see that the condition degenerates if and only if $C$ satisfies the so-called Monge-Amp\`ere equation
$$ C_{xx}(x,z)C_{zz}(x,z)=(C_{xz}(x,z))^2.$$
From the general theory of such equations, we infer that the quadrant $(0,\infty)^2$ can be foliated, i.e., split into a union of pairwise disjoint line segments along which $C$ is linear and the first-order partial derivatives of $C$ are constant. In what follows, we will assume that these segments have negative slope;  see Figure \ref{foliation} below. 

\begin{figure}[htbp]
\begin{center}
\includegraphics[scale=0.7]{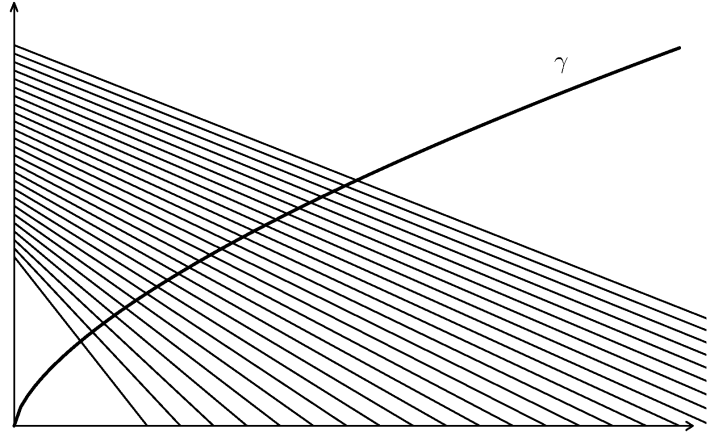}
\caption{
The foliation of $(0,\infty)^2$}\label{foliation}
\end{center}
\end{figure} 

Next, by 2$^\circ$, we  have
\begin{equation}\label{majC}
 C(x,z)\leq xz+\frac{K x^p}{p}+\frac{z^{p'}}{p'}.
\end{equation}
Note that the right-hand side enjoys the following homogeneity property: if we fix $\lambda>0$ and  multiply $x$ by $\lambda^{1/p}$ and $z$ by $\lambda^{1/p'}$, then the whole expression is multiplied by $\lambda$. It seems plausible to assume that the same is true for the left-hand side: $C(\lambda^{1/p}x,\lambda^{1/p'}z)=\lambda C(x,z)$. Finally, if $K$ is the optimal constant, then there should be a nonzero point $(x_0,z_0)$ at which both sides of \eqref{majC} are equal; by the aforementioned homogeneity, such point gives rise to the whole `equality curve' $\gamma=\{(x,z):z=s_0x^{p-1}\}$, where $s_0=z_0x_0^{1-p}$. Note that by \eqref{majC}, the first-order partial derivatives of the functions $C$  and $(x,z)\mapsto xz+{K x^p}/{p}+{z^{p'}}/{p'}$ must match at each point from $\gamma$.

\smallskip

\emph{Step 3. The formula for $C$.} Pick a point $(x,s_0x^{p-1})$ from the equality curve $\gamma$. Let $I$ be the line segment of the foliation passing through this point. If $\alpha=\alpha(x)$ is the slope of this segment, we may write
$$ C(x+d,s_0x^{p-1}+\alpha d)=C(x,s_0x^{p-1})+C_x(x,s_0x^{p-1})d+C_z(x,s_0x^{p-1})\alpha d.$$
But $C(x,s_0x^{p-1})=x^p(s_0+K/p+s_0^{p'}/p')$, since $(x,s_0x^{p-1})\in \gamma$. Moreover, from the last sentence of the previous step, we know that $C_x(x,s_0x^{p-1})=(K+s_0)x^{p-1}$ and $C_z(x,s_0x^{p-1})=x+(s_0x^{p-1})^{p'-1}=x(1+s_0^{1/(p-1)})$. Furthermore, recall that $C_z$ is constant along $I$ and, by \eqref{diff_B}, $ C_z(x,0)=x$. Comparing the latter two expressions for $C_z$, we obtain that $I$ intersects the $x$ axis at the point $(x(1+s_0^{1/(p-1)}),0)$ and hence $\alpha=-s_0^{(p-2)/(p-1)}x^{p-2}$. Putting all the above facts together and calculating a little bit, we get the following explicit (or rather implicit) formula for $C$:
$$ C(x+d,s_0x^{p-1}-s_0^{(p-2)/(p-1)}x^{p-2}d)\!=\!\left(s_0+\frac{K}{p}+\frac{s_0^{p'}}{p'}\right)x^p+(K-s_0^{(p-2)/(p-1)})x^{p-1}d.$$
It remains to guess $K$ and $s_0$. To this end, plug $d=s_0^{1/(p-1)}x$ to obtain
$$ C\big(x(1+s_0^{1/(p-1)}),0\big)=\left(\frac{K}{p}+\frac{s_0^{p'}}{p'}+Ks_0^{1/(p-1)}\right)x^p$$
and hence, differentiating both sides with respect to $x$,
$$ (1+s_0^{1/(p-1)})C_x\big(x(1+s_0^{1/(p-1)}),0\big)=p\left(\frac{K}{p}+\frac{s_0^{p'}}{p'}+Ks_0^{1/(p-1)}\right)x^{p-1}.$$
But $C_x\big(x(1+s_0^{1/(p-1)}),0\big)=C_x(x,s_0x^{p-1})=(K+s_0)x^{p-1}$ (again, by the constancy of partial derivatives along the segments of foliation). Combining the last two observations, we get the equation
$$ p\left(\frac{K}{p}+\frac{s_0^{p'}}{p'}+Ks_0^{1/(p-1)}\right)=(1+s_0^{1/(p-1)})(K+s_0),$$
or, equivalently,
$$ K=\frac{s_0^{(p-2)/(p-1)}+(2-p)s_0}{p-1}.$$
The right-hand side, considered as a function of $s_0$, attains its minimum $(p-1)^{-p}$ at the point $(p-1)^{1-p}$. Thus, it is natural to set $K=(p-1)^{-p}$ and $s_0=(p-1)^{1-p}$ as these extremal values. This leads to the function $B$ studied in the previous section, but to check this, one has to carry out some lengthy calculations. We will give a brief sketch. The above analysis gives the formula for $C$, given implicitly as
$$ C\left(x+d,\left(\frac{x}{p-1}\right)^{p-1}-\left(\frac{x}{p-1}\right)^{p-2}d\right)=p(p-1)^{-p}x^p+p(p-1)^{-p}(2-p)x^{p-1}d.$$
Set $X=x+d$ and $Z=\left(\frac{x}{p-1}\right)^{p-1}-\left(\frac{x}{p-1}\right)^{p-2}d$. One checks directly by \eqref{def_phi<2} that 
$$ \varphi(X^{1-p}Z)=\frac{x+d-pd}{(x+d)(p-1)}.$$
Furthermore, calculating a little bit, we get
\begin{align*}
C(X,Z)\!-\!XZ=\!XZ\!\left[\frac{p-1}{p}\cdot \frac{x+d-pd}{(x+d)(p-1)}+\frac{2-p}{p(p-1)}+\frac{1}{p(p-1)}\frac{(x+d)(p-1)}{x+d-pd}\right],
\end{align*}
and it remains to note that the right-hand side is $B(X,Z)$.

\subsection{The case $p>2$} The reasoning is similar to that in the previous case, so we will be brief. We start with some additional assumptions on the partial derivatives of $B$.  As previously, we work under the condition $B_{xz}\geq 0$ on $(0,\infty)^2$. Furthermore, we impose the vanishing requirement on one of the first-order derivatives: in contrast to the case $1<p<2$, now we assume that $B_x$ is zero on the $z$ axis: $B_x(0,z)=0$. Next, we consider the function $C(x,z)=\tilde{B}(x,z)+xz$ and note that the degeneration of \eqref{discriminant} can be rewritten as the Monge-Ampere equation $ C_{xx}(x,z)C_{zz}(x,z)=(C_{xz}(x,z))^2.$ We assume that the foliation is of the same shape as in the case $1<p<2$; see Figure \ref{foliation}. Next we repeat, word by word, the analysis which leads to the equality curve $\gamma$ and  due to the assumption that $p>2$, this time it is a graph of a \emph{convex} function.

Some substantial differences occur when we turn to the identification of the explicit formula for $C$. As before,  we fix a point $(x,s_0x^{p-1})\in \gamma$, denote by $\alpha$ the slope of the corresponding leaf of foliation and write 
\begin{equation}\label{form_C}
 C(x+d,s_0x^{p-1}+\alpha d)=C(x,s_0x^{p-1})+C_x(x,s_0x^{p-1})d+C_z(x,s_0x^{p-1})\alpha d.
\end{equation}
Now, we have $C(x,s_0x^{p-1})=x^p(s_0+K/p+s_0^{p'}/p')$ (because $(x,s_0x^{p-1})\in \gamma$) and  $C_x(x,s_0x^{p-1})=(K+s_0)x^{p-1}$ and $C_z(x,s_0x^{p-1})=x+(s_0x^{p-1})^{p'-1}=x(1+s_0^{1/(p-1)})$. Now, the assumption $B_x(0,z)=0$ implies $C_x(0,z)=z$. Since $C_x$ is constant along the leaves of foliation, if we compare this to $C_x(x,s_0x^{p-1})$, we get that the line segment $I$ intersects the $z$ axis at the point $(0,(K+s_0)x^{p-1})$. Therefore, we have $\alpha=-Kx^{p-2}$.  Plugging all this information into \eqref{form_C} gives
$$C(x+d,s_0x^{p-1}-Kx^{p-2}d)=\left(\frac{K}{p}+\frac{s_0^{p'}}{p'}+s_0\right)x^p-(Ks_0^{1/(p-1)}-s_0)x^{p-1}d.$$
To guess $K$ and $s_0$, we compute the derivative $C_z$ along the $z$ axis. By the above formula for $C$, we have $ C(0,(K+s_0)x^{p-1})=\left(\frac{K}{p}+\frac{s_0^{p'}}{p'}+Ks_0^{1/(p-1)}\right)x^p$ 
and hence
$$ (K+s_0)C_z(0,(K+s_0)x^{p-1})=\frac{p}{p-1}\left(\frac{K}{p}+\frac{s_0^{p'}}{p'}+Ks_0^{1/(p-1)}\right)x.$$
On the other hand, $C_z$ is constant along $I$, so $C_z(0,(K+s_0)x^{p-1})=C_z(x,s_0x^{p-1})=x(1+s_0^{1/(p-1)})$. Combining this with the previous equation yields an identity which is equivalent to
$$ K=\frac{s_0(p-1)}{s_0^{1/(p-1)}-p+2}.$$
The right-hand side, considered as a function of $s_0\in (p-2,\infty)$, attains its minimal value $(p-1)^p$ at the point $(p-1)^{p-1}$. Setting $K=(p-1)^p$, $s_0=(p-1)^{p-1}$, once can check that the function $(x,z)\mapsto C(x,z)-xz$ is the special function we used in the Section 3. Indeed,  substituting the above values of $K$ and $s_0$ into the formula for $C$, we get
$$ C\Big(x+d,((p-1)x)^{p-1}-(p-1)^2((p-1)x)^{p-2}d\Big)=p'((p-1)x)^p-p(p-2)((p-1)x)^{p-1}d.$$
If we set $X=x+d$ and $Z=((p-1)x)^{p-1}-(p-1)^2((p-1)x)^{p-2}d$, we check directly from \eqref{def_phi>2} that  $ \varphi(X^{1-p}Z)=(px-x-d)/(x+d)$ and
$$ C(X,Z)-XZ=\left(1-\frac{1}{p}\right)XZ\left[\frac{px-x-d}{x+d}+\frac{1}{\frac{px-x-d}{x+d}-p+2}\right].$$
The right-hand side is precisely $B(X,Z)$ and this  shows that the function we discovered coincides with that used in Section 3.

\end{document}